\documentclass[11pt]{article}
\usepackage{amsfonts}
\usepackage{booktabs,graphics,mathrsfs,graphicx,times, fancyhdr,rotating,booktabs,longtable, amsmath,amsfonts,amssymb,cite,balance,color}

\newtheorem{theorem}{Theorem}[section]

\newtheorem{problem}[theorem]{Problem}

\newenvironment{proof}[1][Proof]{\textbf{#1.} }{\ \hfill \rule{0.5em}{0.5em}}

\def\RR{\mathbb{R}}
\def\EE{\mathbb{E}}

\def\bft{{\bf t}}

\def\de{{\delta}}
\def\la{{\lambda}}
\def\si{{\sigma}}

\def\al{{\alpha}}

\def\Ga{{\Gamma}}

\def\de{{\delta}}

\def\si{{\sigma}}

\def\la{{\lambda}}

\def\vare{{\varepsilon}}
\def \eref#1{\hbox{(\ref{#1})}}

\def\EE{\mathbb{ E}\ }
\def \eref#1{\hbox{(\ref{#1})}}

\def\si{{\sigma}}
\def\al{{\alpha}}
\def\bfY{{\bf Y}}
\def\bft{{\bf t}}
\def\bfB{{\bf  B}}
\def\bfN{{\bf  N}}
\def\var{{\hbox{Var}}}

\begin{document}

{\bf EXACT MAXIMUM LIKELIHOOD ESTIMATOR FOR DRIFT \\ FRACTIONAL
BROWNIAN MOTION AT DISCRETE OBSERVATION} \vspace{0.5cm}

\centerline{BY YAOZHONG
HU\renewcommand{\thefootnote}{\fnsymbol{footnote}}\setcounter{footnote}{0}
\footnote{Corresponding author:
hu@math.ku.edu}\renewcommand{\thefootnote}{\arabic{footnote}}\setcounter{footnote}{0}
\footnote{Supported by NSF Grants DMS-0504783}, WEILIN XIAO$^2$ AND
WEIGUO ZHANG\footnote{Supported by National Natural Science Funds
for Distinguished Young Scholar 70825005 \newline  {\sl AMS 2000
subject classifications:} Primary 62G05;  secondary 60H07
\newline {\sl keywords and phrases:} Maximum likelihood estimation,
fractional Brownian motions, discrete observation, strong
consistence, central limit theorem, Malliavin calculus. }}
\vspace{2mm}
\centerline{\sl University of Kansas and South China
University of Technology}
\begin{abstract}
This paper deals with the problems of consistence and strong
consistence of the maximum likelihood estimators of the mean and
variance of the drift fractional Brownian motions   observed at
discrete time instants.  A central limit theorem for these
estimators is also obtained by using the Malliavin calculus.
\end{abstract}

\vspace{5mm} \setcounter{section}{1} \normalsize {\bfseries 1.
Introduction.}
 Long memory processes have been widely applied to various fields,
such as finance, hydrology, network traffic analysis and so on.
Fractional Brownian motions are one special class of long memory
processes when the Hurst parameter $H>1/2$.  The stochastic calculus
for these processes has now been well-established (see \cite{bhoz}).
When a long memory model is used to describe some phenomena, it is
important to identify the parameters in the model. In this paper, we
shall consider the following simple model
\begin{equation}
 Y_t = \mu  t +
\sigma B_t^{H}, \quad t\ge 0\,,\label{e.1.1}
\end{equation}
where $\mu$ and  $\sigma$ are constants to be estimated from
discrete  observations of the process $Y$.  Our method works for
fractional Brownian motions of all parameters. So in this paper we
assume that
 $(B_{t}^{H}, t\ge 0) $ is a fractional Brownian motion of Hurst parameter $H\in
 (0, 1)$. But we do not discuss the case $H=1/2$, the standard
 Brownian motion case since it is known.
 This  means, $(B_{t}^{H}, t\ge 0) $  is a mean $0$ Gaussian process
 with the following covariance structure:
 \[
 \EE\left(B_{t}^{H} B_{s}^{H})\right)=\frac12 \left(t^{2H}+s^{2H}-|t-s|^{2H}\right)\,.
 \]
We assume that the process is observed at discrete time instants
$(t_1, t_2, \cdots, t_N)$. To simplify notation we assume $t_k=kh,
k=1, 2, \cdots, N$ for some fixed length $h>0$. Thus the observation
vector is ${\bf Y}=(Y_{t_{1}},Y_{t_{2}},\cdots,Y_{t_{N}})'$.   We
will obtain the maximum likelihood estimators $\hat \mu_N$ and $\hat
\si_N^2$ of $\mu$ and $\si^2$ respectively and study their
asymptotic behaviors. In particular,  the almost sure convergence
and the  central limit type theorem.

The first reason we chose to study   \eref{e.1.1}  is because it is
simple and we can obtain explicit estimators. The second reason is
that it is also widely applied in various fields. The logarithm of a
widely used geometric fractional Brownian motion, which is popular
in finance,   is of the form \eref{e.1.1}.  This paper  is also
complementary to the work \cite{hu}, where the parameter estimation
problem (with  continuous time observation) for fractional
Ornstein-Uhlenbeck processes is studied.

The parameter estimation problem for long memory processes have been
well-studied (see \cite{beran}, \cite{fox}, \cite{hannan},
\cite{palma}, \cite{privault}).   Although most work requires the
process to be stationary, we may still adapt their idea to analyze
above model \eref{e.1.1}. But we shall use the method of \cite{hu}
which seems to be the simplest  one to us. This method is based on a
result of (\cite{nualart})   and uses the idea of  Malliavin
calculus.

We introduce notation
\begin{equation}
{\bf Y}=\mu {\bf t}+\sigma \bfB_t^{H}\,, \label{e.1.2}
\end{equation}
where and for the rest of the paper ${\bf t}=(h, 2h, \cdots, Nh)'$
and ${\bfB}_t^{H}=(B_{h}^H, \cdots, B_{Nh}^H)' $. The joint
probability density function of ${\bf Y}$ is
\begin{eqnarray*}
h(\bfY) &=&(2 \pi\si^2 )^{-\frac{N}{2}} |\Gamma_H|^{-\frac{1}{2}}
\exp\Big(-\frac{1}{2\si^2 }(\bfY-\mu {\bf t})' \Gamma_H^{-1}
(\bfY-\mu {\bf t})\Big)\,,
\end{eqnarray*}
 where
\[
\Gamma_{H}=\Big[[Cov[B^{H}_{ih},B^{H}_{jh}]\Big]_{i,j=1,2,\cdots,N}=\frac12
h^{2H} \left(i^{2H}+j^{2H}-|i-j|^{2H}\right) _{i,j=1,2,\cdots,N}\,.
\]
The maximum likelihood estimators of $\mu$ and $\sigma^{2}$ from the
observation $\bfY$ are given by
\begin{eqnarray}
        \hat{\mu}
        &=&\frac{{\bf t}'\Gamma_{H}^{-1}\bfY}{{\bf
t}'\Gamma_{H}^{-1} {\bf t}}\,, \label{e.1.3} \\
\hat{\sigma}^{2} &=&\frac{1}{N}\frac{(\bfY'\Gamma^{-1}_{H}\bfY)
({\bf t}'\Gamma^{-1}_{H}{\bf t})- ({\bf t}'\Gamma^{-1}_{H}
\bfY)^{2}}{{\bf t}'\Gamma^{-1}_{H}{\bf t}}\,.\label{e.1.4}
\end{eqnarray}
In Section 2, we shall show that  $ \hat{\mu}$ and
$\hat{\sigma}^{2}$ converge to $\mu$ and $\si^2$ both in mean square
and almost surely. In Section 3, we prove central limit type
theorem. In Section 4, we give some simulation to demonstrate our
estimators
 $ \hat{\mu}$ and $\hat{\sigma}^{2}$.

\setcounter{equation}{0}

\vspace{6mm} \setcounter{section}{2} \normalsize {\bfseries 2.
Consistence.} In this  section we will consider the $L^{2}$
consistency and  the strong consistency of both MLE $\mu$ and
$\sigma^{2}$.

Now, let us first consider the $L^{2}$ consistency of \eref{e.1.3}.

\begin{theorem} The estimator $\hat\mu$ (defined by \eref{e.1.3}) of $\mu$   is unbiased and
it converges in probability to $\mu$ as $N\rightarrow \infty$.
\end{theorem}

\begin{proof}     Substituting  $\bfY$ by $\mu \bft+\si \bfB_t^H$ in \eref{e.1.3}, we have
\begin{equation}
\hat \mu=\mu +\si \frac{\bft '\Ga_H^{-1} \bfB_t^H}{\bft '\Ga_H^{-1}
\bft}\,. \label{e.2.1}
\end{equation}
Thus $\EE\left[\hat \mu\right]=\mu$ and hence $\hat \mu$  is
unbiased.
 On the other hand, we have
\begin{eqnarray*}
\var[\hat{\mu}] &=& \sigma^{2} E \Big[ \frac{{\bf t}'
\Gamma_{H}^{-1} \bfB_t^{H} (\bfB_t^{H})' \Gamma_{H}^{-1} {\bf
t}}{({\bf t}' \Gamma_{H}^{-1} {\bf t} )^{2}}\Big] =\sigma^{2}
\frac{{\bf t}' \Gamma_{H}^{-1} \Gamma_{H} \Gamma_{H}^{-1} {\bf
t}}{({\bf t}' \Gamma_{H}^{-1} {\bf t} )^{2}} = \frac{\sigma^{2}}{
{\bf t}' \Gamma_{H}^{-1} {\bf t} } \,. \notag
\end{eqnarray*}
Denote
\[
M=(m_{ij})_{i,j=1,\ldots,N}\,,\quad {\rm where}\quad
m_{ij}=\frac{1}{2} (i^{2H}+j^{2H}- |i-j|^{2H})\,,
\]
and denote by   $m_{i,j }^{-1}$   the entry of the inverse matrix
$M^{-1}$ of $M$.  Then  we may write
\begin{equation}
\var[\hat{\mu}]=h^{-2H} \frac{\sigma^{2}}{{\bf t}' M^{-1} {\bf t}
}=h^{-2H} h^{-2} \frac{\sigma^{2}}{\sum_{i,j=1}^{N} i j m_{i,j
}^{-1}}= \frac{\sigma^{2}h^{-2H-2}}{\sum_{i,j=1}^{N} i j m_{i,j
}^{-1}} \,. \notag
\end{equation}
We shall use  the following  inequality  (with $x=\bfN=(1, 2,
\cdots, N)$)
\begin{equation}
x'M^{-1}x\geq \frac{\parallel x
\parallel_{2}^{2}}{\lambda_{max}}\,, \notag
\end{equation}
\noindent where $\lambda_{max}$ is the largest eigenvalue of the
matrix $M$. Thus  we  have
\begin{equation}
\var[\hat{\mu}]\leq \sigma^{2} h^{-2H-2}
\frac{\lambda_{max}}{\parallel \bfN
\parallel_{2}^{2}} \notag
\end{equation}
Since $\|\bfN\|^2=1^{2}+2^{2}+\ldots+n^{2}=\frac{n(n+1)(2n+1)}{6}$
we know that $\parallel \bfN \parallel_{2}^{2}\approx N^3$. On  the
other  hand we have  by the Gerschgorin Circle Theorem (see
\cite{golub}, Theorem 8.1.3)
\begin{equation}
\lambda_{max} \leq
%\max_{i=1,\ldots, N} (r_{i}+d_{i})\sim
\max_{i=1,\ldots, N} \sum_{j=1}^{N} \mid m_{ij} \mid \leq CN^{
2H+1}\,,  \notag
\end{equation}
where $C$ a positive constant whose value may be different in
different occurrences.  Consequently, we have
\begin{equation*}
\var[\hat{\mu}]\leq C \sigma^{2} h^{-2H-2} N^{-3} N^{2H+1} =C
 N^{2H-2}\,.
\end{equation*}
which converges  to zero as  $N\rightarrow\infty$.
 \end{proof}

Next we   study  the estimator $\hat \si^2 $ defined by
\eref{e.1.4}.

\begin{theorem}  We have
\begin{equation}
\EE(\hat\si^2)=\frac{N-1}{N} \si^2\quad{\rm and}\quad \var[{\hat
\sigma^{2}}]\xrightarrow{N\rightarrow\infty}0. \label{e.2.2}
\end{equation}
\end{theorem}

\begin{proof}
By replacing $\bfY$ with $\mu \bft+\si \bfB_t^H$  in \eref{e.1.4},
we have
\begin{equation*}
{\hat \sigma^{2}}=\frac{\sigma^{2}}{N} [(\bfB_t^H)' \Gamma^{-1}_{H}
\bfB_t^H - \frac{({\bf t}' \Gamma^{-1}_{H} \bfB_t^H)^{2} }{{\bf
t}^{t }\Gamma^{-1}_{H} {\bf t}} ]\,.
\end{equation*}
Thus
\begin{eqnarray}
E[{\hat \sigma^{2}}] &=&\frac{\sigma^{2}}{N} E [(\bfB_t^H)'
                     \Gamma^{-1}_{H} \bfB_t^H - \frac{({\bf t}' \Gamma^{-1}_{H}
                     \bfB_t^H)^{2} }{{\bf t}^{t }\Gamma^{-1}_{H} {\bf t}} ]\notag \\
                   &=&\frac{\sigma^{2}}{N} (N - \frac{{\bf t}' \Gamma^{-1}_{H}
                     E [\bfB_t^H (\bfB_t^H)']  \Gamma^{-1}_{H} t }{{\bf t}^{t }\Gamma^{-1}_{H} {\bf t}} )
                                         =\frac{N-1}{N}\sigma^{2}\,.\label{e.2.3}
\end{eqnarray}
To compute the variance of ${\hat \sigma^{2}}$  we also need to
compute
  $ E[({\hat \sigma^{2}})^{2}]$:
\begin{align}
 &E[({\hat \sigma^{2}})^{2}] \notag\\
=&\frac{\sigma^{4}}{N^{2}} E
\Big[ \big ((\bfB_t^H)' \Gamma^{-1}_{H} \bfB_t^H - \frac{({\bf t}'
\Gamma^{-1}_{H} \bfB_t^H)^{2} }{{\bf t}^{t
}\Gamma^{-1}_{H} {\bf t}} \big)^{2} \Big] \notag \\
                        =&\frac{\sigma^{4}}{N^{2}}  \Big (
                        E[ ((\bfB_t^H)'\Gamma^{-1}_{H} \bfB_t^H)^{2}]
                        -2 E [(\bfB_t^H)' \Gamma^{-1}_{H}
                        \bfB_t^H \frac{({\bf t}' \Gamma^{-1}_{H} \bfB_t^H)^{2} }
                        {{\bf t}^{t }\Gamma^{-1}_{H} {\bf t}}] \notag \\
&+E[(\frac{({\bf t}' \Gamma^{-1}_{H} \bfB_t^H)^{2} }
                        {{\bf t}^{t }\Gamma^{-1}_{H} {\bf t}})^{2}]
                        \Big)\notag \\
                         =&\frac{\sigma^{4}}{N^{2}}  \Big (
                        E[ ((\bfB_t^H)'\Gamma^{-1}_{H} \bfB_t^H)^{2}]
                        -2 E [(\bfB_t^H)' \Gamma^{-1}_{H}
                        \bfB_t^H \frac{({\bf t}' \Gamma^{-1}_{H} \bfB_t^H)^{2} }
                        {{\bf t}^{t }\Gamma^{-1}_{H} {\bf t}}]+3
                        \Big)\,. \label{e.2.4}
\end{align}
Denote $X=\Ga_H^{-1/2} \bfB_t^H$.  Then $\EE(XX')= E(
\Ga_H^{-1/2}\bfB_t^H(\bfB_t^H )' \Ga_H^{-1/2})=I$.   Therefore, $X$
is a standard Gaussian vector of dimension $N$. For any $\la$ small
enough and $\vare\in \RR$ let us compute the following .
\begin{eqnarray*}
E[\exp (\lambda (\bfB_t^H)'\Gamma^{-1}_{H} \bfB_t^H+ \varepsilon
{\bf t}'\Gamma^{-1}_{H} \bfB_t^H)]&=&E[\exp(\lambda {|\bf X|}^{2}
+\varepsilon {\bf
t}'\Gamma^{-\frac{1}{2}}_{H} {\bf X} )]       \\
                  &=&\frac{1}{(2\pi)^{\frac{N}{2}}}\int_{\RR^{N}}
       e^{-\frac{{|\bf X|}^{2}}{2}+ \lambda {\bf X}^{2} +\varepsilon {\bf
t}'\Gamma^{-\frac{1}{2}}_{H} {\bf X}  }d{\bf X}\,.
                \end{eqnarray*}
A standard technique of completing the squares yields
\begin{equation*}
E[\exp (\lambda (\bfB_t^H)'\Gamma^{-1}_{H} \bfB_t^H+ \varepsilon
{\bf t}'\Gamma^{-1}_{H} \bfB_t^H)]
  =(1-2\lambda)^{-\frac{N}{2}} \exp \left\{ \frac{\varepsilon^{2} {\bf
t}'\Gamma^{-1}_{H} {\bf t}}{2(1-2\lambda)}\right\}=:f(\la, \vare)
\,.
\end{equation*}
We are only interested in the coefficient of $\la^2$ and
$\la\vare^2$ in the above expression $f(\la, \vare)$. We have
\begin{eqnarray*}
f(\la, \vare) &=&
(1+N\la+N(N+2)\la^2+\cdots)\left[1+\frac{\varepsilon^{2} {\bf
t}'\Gamma^{-1}_{H} {\bf t}}{2 }(1+2\la+\cdots)+\cdots\right]\\
&=& 1+N\la+N(N+2)\la^2+\cdots+(N+2)\la\vare^2t'\Ga_H^{-1}t+\cdots\,.
\end{eqnarray*}
Comparing the coefficients of $\la^2$ and $\lambda\varepsilon^{2}$
we have
\begin{eqnarray}
\EE\left[(\bfB_t^H)'\Gamma^{-1}_{H} \bfB_t^H)^2\right]
&=& N(N+2)\,,\label{e.2.5}\\
E[(\bfB_t^H)'\Gamma^{-1}_{H} \bfB_t^H({\bf t}'\Gamma^{-1}_{H}
\bfB_t^H)^{2}]&=&(N+2)({\bf t}'\Gamma^{-1}_{H}{\bf t})\,. \notag
\end{eqnarray}
Hence, we have
\begin{equation}
E [(\bfB_t^H)' \Gamma^{-1}_{H}
                        \bfB_t^H \frac{({\bf t}' \Gamma^{-1}_{H} \bfB_t^H)^{2} }
                        {{\bf t}^{t }\Gamma^{-1}_{H} {\bf t}}]
                        =\frac{(N+2) ({\bf t}'\Gamma^{-1}_{H}{\bf t})}{{\bf t}^{t
                        }\Gamma^{-1}_{H} {\bf t}}=N+2 \,.\label{e.2.6}
\end{equation}
Using \eref{e.2.3}, \eref{e.2.4}, \eref{e.2.5} and \eref{e.2.6}, we
obtain
\begin{eqnarray}
\var[{\hat \sigma^{2}}] &=& E[({\hat \sigma^{2}})^{2}]- (E[{\hat
\sigma^{2}}])^{2}\nonumber\\
&=&\frac{\sigma^{4}}{N^{2}}[N(N+2)-2(N+2)+3-(N-1)^{2}]\nonumber\\
&=&\frac{2(N-1)}{N^{2}}\sigma^{4} \,,\label{e.2.7a}
\end{eqnarray}
which is convergent to $0$.   Thus we prove  the theorem.
\end{proof}

Now we can show the strong consistence of the MLE  $\hat{\mu}$ and
${\hat \sigma^{2}}$ as $N\rightarrow\infty$.

\begin{theorem} The estimators $\hat \mu$ and $\hat \si^2$ defined by
\eref{e.1.3} and \eref{e.1.4}, respectively,  are strongly
consistent, that is,
\begin{eqnarray}
&& \hat{\mu} \rightarrow \mu \quad                 a.s. \quad  as \quad N\rightarrow\infty \label{e.2.7}\\
&&{\hat \sigma^{2}} \rightarrow \sigma^{2} \quad   a.s. \quad  as
\quad N\rightarrow\infty \label{e.2.8}
\end{eqnarray}
\end{theorem}
\begin{proof}   Let's prove the convergence for
$\hat{\mu}$ first. We will use a Borel-Cantelli lemma. To this end,
we will show that
\begin{equation}
\sum_{N\geq 1}P\Big (|\hat{\mu}-\mu|>\frac{1}{N^{\epsilon}}\Big )
<\infty \label{e.2.9}
\end{equation}
for some $\epsilon>0$.

Take $0<\epsilon<1-H$. Then from the Chebyshev's inequality and the
Nelson's hypercontractivity inequality \cite{hu1},  we have
\begin{eqnarray*}
P\Big (|\hat{\mu}-\mu|>\frac{1}{N^{\epsilon}}\Big )
&=& N^{2p\vare} E(|\hat \mu-\mu|^p)\le  C_p N^{2p\vare} \left(E(|\hat \mu-\mu|^2) \right)^{p/2}\\
&\le& C_p ' \si^p h^{-(2H+2)p} N^{2p\vare+(2H-2)p} \,.
\end{eqnarray*}
For sufficiently large $p$, we have $2p\vare+(2H-2)p<-1$. Thus
\eref{e.2.9} is  proved, which implies \eref{e.2.7} by
Borel-Cantelli lemma.

In the same way, we can show    \eref{e.2.8}.
\end{proof}

\setcounter{theorem}{0} \setcounter{equation}{0}
\vspace{6mm}\setcounter{section}{3} \normalsize {\bfseries 3.
Asymptotic.} Now we are interested in the central limiting type
theorem for the estimators $\hat\mu$ and $\hat \si^2$. First from
\eref{e.2.1}, it is easy to see that
 \[\sqrt{{\bf t}^{t }\Gamma^{-1}_{H} {\bf
t}}(\hat{\mu}-\mu)\xrightarrow{\mathcal {L}} \mathcal
{N}(0,\sigma^{2})\quad \hbox{ as $N$ tends to infinity}
\]
We want to study $\hat \si^2$
 \begin{theorem}  We have
 \begin{align}
\sqrt{\frac{N}{2}}\left({\hat \sigma^{2}}-\sigma^{2}\right)
\xrightarrow{\mathcal {L}} \mathcal {N}(0,\sigma^{4}) \quad
\hbox{as\ } \quad  N\rightarrow \infty \,. \label{e.3.1}
\end{align}
\end{theorem}

\begin{proof}  To simplify notation we assume $H>1/2$.  The case $H<1/2$ is similar.
We define
\begin{equation*}
G_{N}=\sqrt{\frac{N}{2}}({\hat
\sigma^{2}}-\sigma^{2})=\frac{\sigma^{2}}{\sqrt{2N}}[(\bfB_t^{H})'
\Gamma^{-1}_{H} \bfB_t^{H} - \frac{({\bf t}' \Gamma^{-1}_{H}
\bfB_t^{H})^{2} }{{\bf t}^{t }\Gamma^{-1}_{H} {\bf
t}}]-\sqrt{\frac{N}{2}}\sigma^{2}\,.
\end{equation*}
From \eref{e.2.7a}, it is obvious that $E[G_{N}^{2}]$ converges to
$\si^4$. Thus from Theorem 4 of \cite{nualart} to show \eref{e.3.1},
it suffices to show that   $\parallel DG_{N}\parallel^{2}_{\mathcal
{H}} \xrightarrow{L_2(\Omega)} C $.

First, using the definition of Malliavin calculus, we obtain
\begin{equation*}
 D_{s}G_{N} =\sqrt{\frac{2}{N}} \si ^2 [D_{s}(\bfB_t^{H})' \Gamma^{-1}_{H} \bfB_t^{H} - \frac{ {\bf t}'
\Gamma^{-1}_{H}\bfB_t^{H} \cdot {\bf t}' \Gamma^{-1}_{H} D_{s}
\bfB_t^{H} }{{\bf t}^{t }\Gamma^{-1}_{H} {\bf t}}]\,,
\end{equation*}
where $D_{s}(\bfB_t^{H})'=(1_{[0,h]}(s),1_{[0,2h]}(s),\ldots,
1_{[0,Nh]}(s))$. Therefore, we have
\begin{align*}
&\parallel D_{s}G_{N} \parallel^{2}_{\mathcal {H}}\\
=&\frac{2\sigma^{4}}{N} \alpha_{H} \int'_{0}\int'_{0} |u-s|^{2H-2}
\left[D_{s}(\bfB_t^{H})' \Gamma^{-1}_{H} \bfB_t^{H}-\frac{ {\bf t}'
\Gamma^{-1}_{H}\bfB_t^{H} \cdot {\bf t}' \Gamma^{-1}_{H} D_{s}
\bfB_t^{H} }{{\bf t}^{t }\Gamma^{-1}_{H}
{\bf t}}\right] \\
&\left[D_{u}(\bfB_t^{H})' \Gamma^{-1}_{H} \bfB_t^{H}-\frac{ {\bf t}'
\Gamma^{-1}_{H}\bfB_t^{H} \cdot {\bf t}' \Gamma^{-1}_{H} D_{u}
\bfB_t^{H} }{{\bf t}^{t }\Gamma^{-1}_{H}
{\bf t}}\right] duds\\
=&\frac{2\sigma^{4}}{N} \cdot 4\alpha_{H}\int'_{0}\int'_{0}
|u-s|^{2H-2} [D_{s}(\bfB_t^{H})' \Gamma^{-1}_{H} \bfB_t^{H} \cdot
D_{u}(\bfB_t^{H})'
\Gamma^{-1}_{H} \bfB_t^{H}\\
&-\frac{ 2 D_{s} (\bfB_t^{H})' \Gamma^{-1}_{H} \bfB_t^{H} \cdot {\bf
t}' \Gamma^{-1}_{H}\bfB_t^{H} \cdot {\bf t}^{t }\Gamma^{-1}_{H}
D_{u}\bfB_t^{H}}{{\bf t}^{t }\Gamma^{-1}_{H}
{\bf t}}\\
&+\frac{({\bf t}' \Gamma^{-1}_{H}\bfB_t^{H})^{2} \cdot {\bf t}'
\Gamma^{-1}_{H}D_{s}\bfB_t^{H} \cdot {\bf t}'
\Gamma^{-1}_{H}D_{u}\bfB_t^{H} }{({\bf t}^{t }\Gamma^{-1}_{H}
{\bf t})^{2}}]duds \\
=&2 \sigma^{4}[A_{T}^{(1)}-2A_{T}^{(2)}+A_{T}^{(3)}].
\end{align*}
Since both $D_{s} (\bfB_t^{H})' \Gamma^{-1}_{H} \bfB_t^{H}$ and
$D_{u} (\bfB_t^{H})' \Gamma^{-1}_{H} \bfB_t^{H}$ are Gaussian random
variables we can write
\begin{align*}
 &E(|A_{T}^{(1)}-E A_{T}^{(1)}|^{2})\\
=&\frac{2}{N^{2}}\al_H^2 \int_{[0,T]^{4}} E[D_{s}(B_{\bf t}^{H})'
\Gamma^{-1}_{H}B_{\bf t}^{H}\cdot D_{r}
(B_{\bf t}^{H})'\Gamma^{-1}_{H}B_{\bf t}^{H}]\\
&\cdot E[D_{u}(B_{\bf t}^{H})' \Gamma^{-1}_{H}B_{\bf t}^{H}\cdot
D_{v} (B_{\bf t}^{H})'\Gamma^{-1}_{H}B_{\bf
t}^{H}]|s-u|^{2H-2}|r-v|^{2H-2} dsdrdudv \\
=&\frac{2}{N^{2}} \int_{[0,T]^{4}} [D_{s}(B_{\bf t}^{H})'
\Gamma^{-1}_{H}D_{r}B_{\bf t}^{H}\cdot D_{u}(B_{\bf t}^{H})'
\Gamma^{-1}_{H}D_{v}B_{\bf t}^{H}]\\
&\cdot |s-u|^{2H-2}|r-v|^{2H-2} dsdrdudv \,.
\end{align*}

Let $\Gamma^{-1}_{H}=(\Gamma^{-1}_{ij})_{i,j=1,\ldots,N}$ ,
$\Gamma_{H}=(\Gamma_{ij})_{i,j=1,\ldots,N}$ and $\delta_{lk} $  be
the Kronecker symbol. We shall use $\int_0^{ih} \int_0^{i'h}
|s-u|^{2H-2} dsdu =\Ga_{ii'}$ and $\sum_{j=1}^N \Ga_{ij}^{-1}
\Ga_{i'j} =\de_{ii'}$. Then  we have
\begin{align*}
&E(|A_{T}^{(1)}-E A_{T}^{(1)}|^{2}) \\
=&\frac{2}{N^{2} } \int_{[0,T]^{4}}
1_{[0,ih]}(s)\Gamma^{-1}_{ij}1_{[0,jh]}(r)\cdot
1_{[0,i'h]}(u)\Gamma^{-1}_{i'j'}1_{[0,j'h]}(v)\\
&\cdot\alpha_{H}|s-u|^{2H-2}\alpha_{H}|r-v|^{2H-2} dsdrdudv\\
=&\frac{2}{N^{2}\ }\sum^{N}_{i,j=1}\sum^{N}_{i',j'=1}
\Gamma^{-1}_{ij} \Gamma^{-1}_{i'j'}\cdot \Gamma_{ii'} \Gamma_{jj'}\\
=& \frac{2}{N^{2} }\sum_{i, j'=1}^N \de_{ij'}^2=\frac{2}{N}\,,
\end{align*}
 which converges to 0 as $N\rightarrow \infty$.

Now we deal with $A_T^{(2)}$.
\begin{align*}
 &E(|A_{T}^{(2)}-E A_{T}^{(2)}|^{2})\\
=&\frac{2\al_H^2}{N^{2}} \int_{[0,T]^{4}} E[D_{s} (\bfB_t^{H})'
\Gamma^{-1}_{H} \bfB_t^{H}\cdot {\bf t}'
\Gamma^{-1}_{H}\bfB_t^{H}\cdot \frac{{\bf t}^{t }\Gamma^{-1}_{H}
D_{u}\bfB_t^{H}}{{\bf t}^{t }\Gamma^{-1}_{H} {\bf t}}]\\
&\cdot E[D_{r} (\bfB_t^{H})' \Gamma^{-1}_{H} \bfB_t^{H}\cdot {\bf
t}' \Gamma^{-1}_{H}\bfB_t^{H}\cdot\frac{{\bf t}^{t }\Gamma^{-1}_{H}
D_{v}\bfB_t^{H}}{{\bf t}^{t }\Gamma^{-1}_{H} {\bf
t}}]\\
&\cdot |s-v|^{2H-2}|u-r|^{2H-2}dvdsdudr\\
=&\frac{2\alpha_H^2 }{N^{2}} \int_{[0,T]^{4}} \frac{D_{s}
(\bfB_t^{H})' \Gamma^{-1}_{H} {\bf t}\cdot{\bf t}^{t
}\Gamma^{-1}_{H} D_{u}\bfB_t^{H}}{{\bf t}^{t }\Gamma^{-1}_{H} {\bf t}}\\
&\cdot \frac{D_{r} (\bfB_t^{H})' \Gamma^{-1}_{H} {\bf t}\cdot{\bf
t}^{t }\Gamma^{-1}_{H} D_{v}\bfB_t^{H}}{{\bf t}^{t }\Gamma^{-1}_{H} {\bf t}} |s-v|^{2H-2}|u-r|^{2H-2}dvdsdudr \\
=&\frac{2}{N^{2} } \int_{[0,T]^{4}} \frac{
\sum^{N}_{i,j=1}\sum^{N}_{i',j'=1} {\bf
1}_{[0,ih]}(s)\Gamma^{-1}_{ij}jh\cdot i'h \Gamma^{-1}_{i'j'}{\bf
1}_{[0,j'h]}(u)}{{\bf t}^{t }\Gamma^{-1}_{H} {\bf
t}}\\
&\cdot\frac{\sum^{N}_{k,l=1}\sum^{N}_{k',l'=1} {\bf
1}_{[0,kh]}(r)\Gamma^{-1}_{kl}lh\cdot k'h \Gamma^{-1}_{k'l'}{\bf
1}_{[0,l'h]}(v) }{{\bf t}^{t }\Gamma^{-1}_{H} {\bf
t}}\\
&\cdot \alpha_{H} |s-v|^{2H-2} \alpha_{H} |u-r|^{2H-2}dvdsdudr \\
=& \frac{2}{N^2} \sum \Ga_{ij}^{-1} jh \Ga_{i'j'} i'h \Ga_{kl}^{-1}
lh \Ga_{k'l'}^{-1} k'h \Ga_{il'} \Ga_{j'k}\,,
\end{align*}
where the summation is over $1\le i, j, i', j', k, l, k', l'\le N$.
Sum first over $1\le i, j'\le N$ and then over $ 1\le l', k\le N$,
we have
\begin{eqnarray*}
E(|A_{T}^{(2)}-E A_{T}^{(2)}|^{2}) &=& \frac{2}{N^2} \frac{\sum_{j,
l, k', i'=1}^N jhlhk'hl'h \Ga_{k'j}^{-1}
\Ga_{i'l}^{-1}}{(t'\Ga_H^{-1}t)^2}=\frac{2}{N^2}\,.
\end{eqnarray*}
which converges to 0 as $N\rightarrow \infty$.

As for $A_T^{(3)}$,  we have
\begin{align*}
 &E(|A_{T}^{(3)}-E A_{T}^{(3)}|^{2})\\
=&\frac{2\al_H^2 }{N^{2}} \int_{[0,T]^{4}} \frac{{\bf t}'
\Gamma^{-1}_{H}D_{s}\bfB_t^{H} \cdot {\bf t}'
\Gamma^{-1}_{H}D_{s'}\bfB_t^{H} }{({\bf t}^{t }\Gamma^{-1}_{H} {\bf
t})^{2}}\cdot \frac{{\bf t}' \Gamma^{-1}_{H}D_{u}\bfB_t^{H} \cdot
{\bf t}' \Gamma^{-1}_{H}D_{u'}\bfB_t^{H}}{({\bf t}^{t
}\Gamma^{-1}_{H}
{\bf t})^{2}} \\
&\cdot \big ( E({\bf t}' \Gamma^{-1}_{H}\bfB_t^{H})^{2}\big)^{2}|s-u|^{2H-2}|s'-u'|^{2H-2}dvdsdudr \\
=&\frac{2 \al_H^2}{N^{2}} \int_{[0,T]^{4}} \frac{{\bf t}'
\Gamma^{-1}_{H}D_{s}\bfB_t^{H} \cdot {\bf t}'
\Gamma^{-1}_{H}D_{s'}\bfB_t^{H}}{({\bf t}^{t }\Gamma^{-1}_{H} {\bf
t})^{2}}\cdot \frac{{\bf t}' \Gamma^{-1}_{H}D_{u}\bfB_t^{H} \cdot
{\bf t}' \Gamma^{-1}_{H}D_{u'}\bfB_t^{H}}{({\bf t}^{t
}\Gamma^{-1}_{H}
{\bf t})^{2}} \\
&\cdot \big ( {\bf t}' \Gamma^{-1}_{H}{\bf t}\big)^{2}|s-u|^{2H-2}|s'-u'|^{2H-2}dvdsdudr \\
=&\frac{2\al_H^2 }{N^{2}} \Big [ \int_{[0,T]^{2}} \frac{{\bf t}'
\Gamma^{-1}_{H}D_{s}\bfB_t^{H} {\bf t}'
\Gamma^{-1}_{H}D_{u}\bfB_t^{H}}{{\bf t}^{t }\Gamma^{-1}_{H} {\bf t}}
|s-u|^{2H-2} dsdu
\Big]^{2}\\
=&\frac{2}{N^{2}} \Big [ \int_{[0,T]^{2}}
\frac{\sum^{N}_{i,j=1}\sum^{N}_{i',j'=1}ih \Gamma^{-1}_{ij} {\bf
1}_{[0,jh]}(s)\cdot i'h \Gamma^{-1}_{i'j'} {\bf 1}_{[0,j'h] }(u)}
{{\bf t}^{t }\Gamma^{-1}_{H} {\bf t}} \\
&\cdot \alpha_{H}|s-u|^{2H-2} dsdu \Big]^{2}\\
=&\frac{2}{N^{2}} \Big [ \frac{\sum^{N}_{i,j=1}\sum^{N}_{i',j'=1}ih
\Gamma^{-1}_{ij} i'h \Gamma^{-1}_{i'j'}}{{\bf t}^{t }\Gamma^{-1}_{H}
{\bf t}}\Gamma_{jj'} \Big]^{2}
 = \frac{2}{N^{2} }\,,
\end{align*}
which converges to 0 as $N\rightarrow \infty$.

By triangular inequality, we have that
\begin{align*}
& E(\parallel D G_{N} \parallel^{2}_{\mathcal {H}}-E\parallel D
G_{N}\parallel^{2}_{\mathcal {H}})^{2}\\
=&E(A_{T}^{(1)}+A_{T}^{(2)}+A_{T}^{(3)}-E(A_{T}^{(1)}+A_{T}^{(2)}+A_{T}^{(3)}))^{2}\\
\leq & 9\big[ E(A_{T}^{(1)}-E(A_{T}^{(1)}))^{2} +
E(A_{T}^{(2)}-E(A_{T}^{(2)}))^{2}+E(A_{T}^{(3)}-E(A_{T}^{(3)}))^{2}\big]
\rightarrow   0.
\end{align*}
 This completes the proof of the theorem.
 \end{proof}

\newpage
\vspace{6mm} \setcounter{section}{4} \normalsize {\bfseries 4.
Simulation.} This section contains  numerical simulations  of the
estimators obtained in this paper.   The fractional Brownian motions
are simulated by the  Paxson's method \cite{paxson}.

\begin{center}
\centerline{TABLE 1}
\begin{tabular}{cccccccccccccc}
 \multicolumn{9}{c}{\sl The means and standard
deviations of estimators ($\mu$=0.7880, $\sigma^{2}$=0.8116)}\\
\toprule%
& \multicolumn{2}{c}{$H$=0.25}& \multicolumn{2}{c}{$H$=0.45}
&\multicolumn{2}{c}{$H$=0.55}&\multicolumn{2}{c}{$H$=0.75} \\
\cmidrule(l{1em}r){2-3} \cmidrule(l{1em}r){4-5}
\cmidrule(l{1em}r){6-7} \cmidrule(l{1em}r){8-9}
&$\mu$ & $\sigma^{2}$ & $\mu$ & $\sigma^{2}$ & $\mu$ & $\sigma^{2}$ & $\mu$ & $\sigma^{2}$ \\
\hline
Mean       &0.7862& 0.8152 & 0.7884 & 0.8153 & 0.7911 & 0.8126   & 0.7678 & 0.7910 \\
Std.dev.   &0.0116& 0.0830 & 0.0112 & 0.0937 & 0.0514 & 0.0692   & 0.0974 & 0.0736 \\
\bottomrule
\end{tabular}
\end{center}

\vskip 0.5cm
\begin{center}
\centerline{TABLE 2}
\begin{tabular}{cccccccccccccc}
 \multicolumn{9}{c}{\sl The means and standard
deviations of estimators ($\mu$=1.5880, $\sigma^{2}$=1.8116)}\\
\toprule%
& \multicolumn{2}{c}{$H$=0.25}& \multicolumn{2}{c}{$H$=0.45}
&\multicolumn{2}{c}{$H$=0.55}&\multicolumn{2}{c}{$H$=0.75} \\
\cmidrule(l{1em}r){2-3} \cmidrule(l{1em}r){4-5}
\cmidrule(l{1em}r){6-7} \cmidrule(l{1em}r){8-9}
&$\mu$ & $\sigma^{2}$ & $\mu$ & $\sigma^{2}$ & $\mu$ & $\sigma^{2}$ & $\mu$ & $\sigma^{2}$ \\
\hline
Mean       &1.5863& 1.8694 & 1.5882 & 1.8719 &1.5961 & 1.8647  & 1.5925 & 1.7864 \\
Std.dev.   &0.0148& 0.1724 & 0.0456 & 0.1786 &0.0710 & 0.1567  & 0.1879 & 0.1644 \\
\bottomrule
\end{tabular}
\end{center}

\vskip 0.5cm
\begin{center}
\centerline{TABLE 3}
\begin{tabular}{cccccccccccccc}
 \multicolumn{9}{c}{\sl The means and standard
deviations of estimators ($\mu$=3.5880, $\sigma^{2}$=5.8116)}\\
\toprule%
& \multicolumn{2}{c}{$H$=0.25}& \multicolumn{2}{c}{$H$=0.45}
&\multicolumn{2}{c}{$H$=0.55}&\multicolumn{2}{c}{$H$=0.75} \\
\cmidrule(l{1em}r){2-3} \cmidrule(l{1em}r){4-5}
\cmidrule(l{1em}r){6-7} \cmidrule(l{1em}r){8-9}
&$\mu$ & $\sigma^{2}$ & $\mu$ & $\sigma^{2}$ & $\mu$ & $\sigma^{2}$ & $\mu$ & $\sigma^{2}$ \\
\hline
Mean       &3.5861& 5.8133 & 3.5810 & 5.8192 &3.5837 & 5.8229   & 3.5834 & 5.8346 \\
Std.dev.   &0.0314& 0.1648 & 0.0792 & 0.1737 &0.0905 & 0.1031   & 0.0526 & 0.1026 \\
\bottomrule
\end{tabular}
\end{center}

\vskip 0.2cm From these  numerical computations, we see the
estimators are excellent both for $H>1/2$ and $H<1/2$.

\section*{Acknowledgements}   We thank David Nualart for helpful discussions.

\scriptsize{
\begin{flushleft}
\textsc{ YAOZHONG HU  \hskip 17em WEILIN XIAO AND WEIGUO ZHANG}\\
\textsc{ Department of Mathematics \hskip 11em School of Business Administration}\\
\textsc{ University of Kansas, 405 Snow Hall \hskip 6.36em South China University of Technology}\\
\textsc{ Lawrence, Kansas \hskip 15.8em GUANGZHOU}\\
\textsc{ USA \hskip 22.8em CHINA}\\
~E-MAIL: hu@math.ku.edu \hskip 14.15em E-MAIL: xiao@math.ku.edu
\end{flushleft}

\end{document}